\documentclass[12pt,a4paper]{amsart}
\usepackage[utf8]{inputenc}

\usepackage{amssymb,amsfonts,amsthm}
\usepackage{amsrefs}
\usepackage{mathtools}
\usepackage{graphicx}
\usepackage[colorlinks=true, allcolors=blue]{hyperref}
\usepackage{csquotes}% for quotes
\usepackage{booktabs}
\usepackage{tikz-cd}
 
\usepackage{enumerate}
\usepackage{bbm}
\usepackage{dsfont}
\usepackage{mathrsfs}
\usepackage{tikz-cd}
\usepackage{import}
\usepackage{hyperref}
\usepackage{tikz}
\usepackage{xspace}
\usepackage{ytableau}

\usepackage{comment}
\usepackage{makecell} % more control over cells in a table

\usepackage[margin=1in]{geometry}
\usepackage[colorinlistoftodos, bordercolor=orange, backgroundcolor=orange!20, linecolor=orange, textsize=scriptsize]{todonotes}

\newtheorem{thm}{Theorem}[section]
\newtheorem{cor}[thm]{Corollary}
\newtheorem{prop}[thm]{Proposition}
\newtheorem{lem}[thm]{Lemma}
\newtheorem{conj}[thm]{Conjecture}

\theoremstyle{definition}

\newtheorem{example}[thm]{Example}

\newtheorem{rmk}[thm]{Remark}

\newtheorem{quest}[thm]{Question}

\def\N{\mathbb{N}}
\def\Z{\mathbb{Z}}
\def\Q{\mathbb{Q}}
\def\R{\mathbb{R}}
\def\C{\mathbb{C}}
\def\T{\mathbb{T}}

\newcommand\cC{\mathcal{C}}
\newcommand\cP{\mathcal{P}}

\newcommand\polymake{\texttt{polymake}\xspace}
\newcommand\OSCAR{\texttt{OSCAR}\xspace}
\newcommand\scip{\texttt{SCIP}\xspace}
\newcommand\MathRepo{\texttt{MathRepo}\xspace}

\newcommand{\U}{\operatorname{U}}
\newcommand{\SU}{\operatorname{SU}}
\newcommand{\Sym}{\operatorname{Sym}}
\newcommand{\diag}{\operatorname{diag}}
\newcommand{\tr}{\operatorname{tr}}
\newcommand{\tdeg}{\operatorname{tdeg}}
\newcommand{\lcm}{\operatorname{lcm}}
\newcommand{\SLalg}{\mathfrak{sl}}

\title[Bergman's compact amalgamation problem]{Some thoughts and experiments on Bergman's compact amalgamation problem}

\author{Michael Joswig}
\address{M.J., TU Berlin and MPI MiS Leipzig, Germany}
\email{joswig@math.tu-berlin.de}

\author{Mario Kummer}
\address{M.K., Technische Universit\"at Dresden, Germany} 
\email{mario.kummer@tu-dresden.de}
\author{Andreas Thom}
\address{A.T., Technische Universit\"at Dresden, Germany} 
\email{andreas.thom@tu-dresden.de}
\author{Claudia He Yun}
\address{C.H.Y., MPI MiS Leipzig, Germany}
\email{clyun@mis.mpg.de}

\subjclass{%
  22C05, % Topological groups, Lie groups: (1973-now) Compact groups
  18B99, % Category theory: (1973-now) None of the above, but in this section 
  90-05, % (2020-now) Experimental work for problems pertaining to operations research and mathematical programming
  90C90 % (1991-now) Applications of mathematical programming 
}

\begin{document}

\begin{abstract}
We study the question whether copies of $S^1$ in $\SU(3)$ can be amalgamated in a compact group. This is the simplest instance of a fundamental open problem in the theory of compact groups raised by George Bergman in 1987. Considerable computational experiments suggest that the answer is positive in this case. We obtain a positive answer for a relaxed problem using theoretical considerations.
\end{abstract}

\maketitle

\section{Introduction}

We write $\SU(3)$ for the group of $3 {\times} 3$ complex, unitary matrices with determinant equal to $1$. Consider the closed subgroups $\T_1 = \{\diag(z,1,z^{-1}) \mid z \in S^1 \}$ and $\T_2 =\{\diag(z,z,z^{-2}) \mid z \in S^1 \}$, where $S^1$ denotes the multiplicative group of complex numbers of norm one. Both $\T_1$ and $\T_2$ are isomorphic to $S^1$ as topological groups, via the natural isomorphisms $z \mapsto \diag(z,1,z^{-1})$ and $z \mapsto \diag(z,z,z^{-2})$, respectively.  However, the two representations of $S^1$ in $\SU(3)$ are not equal and not even conjugate in $\SU(3).$ So it is a natural question to wonder whether there exist unitary representations $\pi_1,\pi_2 \colon \SU(3) \to \U(n)$ for some $n$, such that the two representations of $S^1$ can be matched, or more precisely $$\pi_1(\diag(z,z^{-1},1))=\pi_2(\diag(z,z,z^{-2})), \quad \textrm{ for all } z \in S^1.$$ We denote by $\rho$ the natural representation of $\SU(3)$ on $\C^3$. One can then check that the two $9$-dimensional representations $\pi_1=\rho \otimes \bar \rho$ and $\pi_2 = \rho \oplus \bar \rho \oplus 1^{\oplus 3}$ solve this problem, where $\bar \rho$ denotes the conjugate representation, and $1$ is the trivial one-dimensional representation. We will come back to this example several times in this article.

This concrete question belongs to a more general set of problems that was first studied by Bergman \cite{MR893156}. Let $A,B$ be compact groups that share a common closed subgroup $C$, see \cite{MR4201900} for background on the theory of general compact groups. It is natural to consider the abstract amalgamated free product $G:=A \ast_C B$ and try to study the analytic properties that it inherits from its constituents. A natural question is whether $G$ can carry a possibly non-Hausdorff compact topology that restricts to the given topologies on $A$ and $B.$ Equivalently, we ask whether $A$ and $B$ can be amalgamated over $C$ in the category of compact groups, i.e., if there exists a compact group $D$ and embeddings of $A$ and $B$ in $D$ that agree on the common copy of $C$.
$$\begin{tikzcd}
&  A \arrow[rrd, bend left, "\pi_1"]  &  &   \\
C \arrow[ru, bend left] \arrow[rd, bend right] &  &   & D \\
&  B \arrow[rru, bend right, "\pi_2"] &  &  
\end{tikzcd}
$$
Compact groups carry bi-invariant metrics that generate the topology. Thus, a first obstruction to a positive answer is that $A$ and $B$ may not carry bi-invariant metrics that agree on $C$. If this is the case, a bi-invariant pseudo-metric on $G$ that is faithful on $A$ and $B$ cannot exist. In particular, there does not exist a compact group $D$ as described above. Bergman showed that this type of argument rules out existence of compact amalgams in many cases. It is a fundamental open problem, if amalgamation is always possible for $C=S^1$ or $C=\SU(n)$, see \cite[Question 20]{MR893156}.  %\todo{Here and below: do we need two notations for the same, $\T$ and $S^1$?} 

The purpose of this note is to explore an equivalent algebraic reformulation of the problem in the simplest possible case. What got us started was the strategy outlined after Question 20 in \cite{MR893156}.  That strategy is amenable to standard computer algebra systems.  Our computer experiments, with \scip \cite{scip} and \OSCAR \cites{OSCAR-book,OSCAR}, suggest that solutions to the original problem can always be found but get increasingly complicated. For example, merging the subgroups
$\T_1 = \{\diag(z,z^5,z^{-6}) \mid z \in S^1 \}$ and $\T_2 =\{\diag(z,z^7,z^{-8}) \mid z \in S^1 \}$ of $\SU(3)$ required us to consider all possible direct sums of the first $120$ mutually non-equivalent irreducible representations of $\SU(3)$ (in some specified order) until a pair of unitary representations of $\SU(3)$ on a complex vector space of dimension roughly $300{,}000$ could be found that solves the problem.

\subsection*{Acknowledgements}
We are indebted to Ulrich Thiel for contributing partitions, Schur polynomials and associated Lie theory to \OSCAR; we are also grateful to Matthias Zach for helping with the \OSCAR integration. We thank Tobias Boege for many helpful conversations.
MJ has been supported by Deutsche Forschungsgemeinschaft (DFG, German Research Foundation): \enquote{Symbolic Tools in Mathematics and their Application} (TRR 195, project-ID 286237555); The Berlin Mathematics Research Center MATH$^+$ (EXC-2046/1, project ID 390685689). MK has been supported by Deutsche Forschungsgemeinschaft (DFG, German Research Foundation), grant number 502861109.

\section{Some basic observations}

Let us study the question of amalgamation of the base $C=S^1$.  It follows from the Peter--Weyl Theorem \cite[Theorem 1.12]{Knapp:1986} that in order to construct amalgams of general compact groups $A$ and $B$, it is enough to consider the case $A=\U(n)$, $B=\U(m)$. Since $\U(n)$ embeds to $\SU(n+1)$, we can further restrict to the case $A=\SU(n)$, $B=\SU(m)$. The simplest case that comes to mind is $A=B=\SU(2)$. In this case, however, every embedding of $S^1$ in $\SU(2)$ is conjugate to the map $z\mapsto \diag(z,z^{-1})$. Thus, the first truly non-trivial case might be $A=\SU(2)$ and $B=\SU(3)$ or somewhat more general $A=B=\SU(3)$. Our first task is to describe the possible embeddings of $S^1$ in $\SU(3)$. %and then, since every compact group $D$ has finite-dimensional unitary representations.  
Our second task is to search for pairs of faithful, finite-dimensional, unitary representations of $\SU(3)$ that agree on the embedded copies of $S^1$.

The first task is easy to solve. Up to conjugation in $\SU(3)$, every embedding of $S^1$ into $\SU(3)$ is given by three integers $(a,b,c) \in \Z^3$ such that $a+b+c=0$ and $\gcd(a,b,c)=1$. The embedding associated with the triplet $(a,b,c)$ is concretely given by 
$$\psi_{a,b,c}(z):= \diag(z^{a},z^b,z^c) \in \SU(3), \quad \textrm{ for all } z \in S^1.$$

In order to address the second task, we recall some facts about the finite-dimensional unitary representation theory of $\SU(3)$. Let $\rho$ be the standard representation of $\SU(3)$ on $\C^3$ and $\bar \rho$ be its dual or conjugate. We denote by $\pi_{m,n} \colon \SU(3) \to \U(\Gamma_{m,n})$ the irreducible representation parameterized by the weight $(m,n)$, a pair of non-negative integers. This representation corresponds to the Young tableaux of shape $(m+n,n)$. According to Fulton--Harris \cite[§13.2]{MR1153249}, we have
\begin{equation}
    \label{rep}
    \Gamma_{m,n} = \ker\left(\Sym^m(\rho) \otimes \Sym^n(\bar \rho) \stackrel{\phi_{m,n}}{\to} \Sym^{m-1}(\rho) \otimes \Sym^{n-1}(\rho)\right),
\end{equation} %\todo{$\Gamma_{m,n}$ does not occur before or after}
where $\phi_{m,n}$ denotes the natural (surjective) contraction map
\begin{equation*}
    \phi_{m,n} \bigl((v_1\dots v_m)\otimes(v_1^*\dots v_n^*) \bigr) := \sum_{i=1}^m \sum_{j=1}^n \langle v_i,v_j^*\rangle (v_1\dots \hat{v_i}\dots v_m)\otimes (v_1^*\dots \hat{v}^*_j\dots v^*_n).
\end{equation*}
Note that representations of the compact Lie group $\SU(3)$ correspond to representations of the simple Lie algebra $\SLalg_3\C$; see \cite[§9.3]{MR1153249}.

Every unitary representation $\sigma$ of $\SU(3)$ extends to a unitary representation $\sigma'$ of $\U(3)$, which is however not unique. The character of a unitary representation $\sigma \colon \SU(3) \to U(k)$ is a symmetric polynomial $\chi(\sigma) \in \Z[x_1,x_2,x_3]$ determined uniquely up to some element in the ideal generated by $x_1x_2x_3-1$ by the property $$\chi(\sigma)(x_1,x_2,x_3)= \tr\bigl(\sigma'(\diag(x_1,x_2,x_3))\bigr), \quad \textrm{ for all } x_1,x_2,x_3 \in S^1.$$ It is known that
$\chi(\pi'_{m,n}) = s_{m+n,n},$ where $s_{\lambda}$ denotes the Schur polynomial of the Young tableaux with two parts $(m+n,n).$ Here, $\pi'_{m,n}$ is the representation of $\U(3)$ described by the same formula as in Equation \eqref{rep}.
Recall that $\chi(\sigma)(1,1,1)$ equals the dimension of the representation $\sigma$.

Now, every finite-dimensional unitary representation is a direct sum of irreducible unitary representations, and thus every character of such a representation is a symmetric, non-negative integer linear combination of Schur polynomials. We call such a symmetric polynomial \emph{Schur positive}.
Here, the polynomial $s_{1,1,1}=x_1x_2x_3$ corresponds to the trivial representation $1$.
We say that a Schur positive polynomial is \emph{non-trivial} if it is non-scalar modulo the polynomial $x_1x_2x_3-1.$

%To get rid of the ambiguity we may pass to symmetric Laurent polynomials in three variables and restrict ourselves to linear combinations of monomials $x_1^{v_1}x_2^{v_2}x_3^{v_3}$ with the additional constraint that $v_1+v_2+v_3$ equals $0,1$ or $2$.

Back to the original problem, we are interested in the question if for a given pair of triplets of integers, $(v_1,v_2,v_3)$ and $(w_1,w_2,w_3)$, with $v_1+v_2+v_3=w_1+w_2+w_3=0$ and $\gcd(v_1,v_2,v_3)=\gcd(w_1,w_2,w_3)=1$, we can find a pair of unitary representations $$\sigma_1,\sigma_2 \colon \SU(3) \to \U(k),$$ such that
$$\sigma_1\bigl(\psi_{v_1,v_2,v_3}(z)\bigr)= \sigma_2\bigl(\psi_{w_1,w_2,w_3}(z)\bigr) \quad \textrm{ for all } z \in S^1.$$
However, these two representations are conjugate (and hence equal after conjugation) if and only if the associated characters agree. Hence, we arrive at the equivalent condition
$$\chi(\sigma_1)(z^{v_1},z^{v_2},z^{v_3}) = \chi(\sigma_2)(z^{w_1},z^{w_2},z^{w_3}) \quad \textrm{ for all } z \in S^1.$$
Thus, putting everything in more algebraic terms, the second task amounts to finding Schur positive polynomials $P$ and $Q$ such that the equality
\begin{equation}
  P(z^{v_1},z^{v_2},z^{v_3}) = Q(z^{w_1},z^{w_2},z^{w_3})
\end{equation}
holds in the Laurent polynomial ring $\Q[z^\pm]$.
For brevity, given a vector $v=(a,b,c)$ with $a+b+c=0$, we write $P_v(z) = P(z^a,z^b,z^c)$ for the substitution.

There is an additional subtlety that we did not address so far: unitary representations of $\SU(3)$ need not be injective.
If we pick a non-trivial third root of unity, $\xi=\exp(2\pi i/3)\in S^1$,
then the subgroup
$$ Z = \bigl\langle \diag(\xi,\xi,\xi) \bigr\rangle \cong \Z/3\Z, $$
forms the center of $\SU(3)$.
Note that $Z$ is the only non-trivial normal subgroup of $\SU(3)$, i.e., the quotient $\SU(3)/Z$ is simple.
The following result characterizes the injective unitary representations of $\SU(3)$.
\begin{prop}\label{prop:injective}
Let $\sigma \colon \SU(3) \to \U(k)$ be a unitary representation with character $P=\chi(\sigma') \in \Z[x_1,x_2,x_3]$ for some extension $\sigma'$ of $\sigma$ to $\U(3)$. Then the following conditions are equivalent:
\begin{enumerate}
    \item $\sigma$ is injective,
    \item $P(\xi,\xi,\xi) \neq P(1,1,1)$, and
    \item $P$, written in terms of the Schur basis, has a summand (with positive coefficient) having total degree not divisible by three.
\end{enumerate}
\end{prop}

\begin{proof} Since the center $Z$ of $\SU(3)$ is generated by $\diag(\xi,\xi,\xi)$ and because $Z$ is the only non-trivial normal subgroup of $\SU(3)$, the representation $\sigma$ is injective if and only if $\sigma(\diag(\xi,\xi,\xi))$ is distinct from the $k{\times}k$ unit matrix. The eigenvalues of $\sigma(\diag(\xi,\xi,\xi))$ are complex numbers of modulus $1$ and $P(\xi,\xi,\xi)=\tr(\sigma(\diag(\xi,\xi,\xi)))$ is equal to their sum. Hence, $P(\xi,\xi,\xi)=k=P(1,1,1)$ if and only if all these eigenvalues are equal to $1$ if and only if $\sigma$ is not injective. This proves the equivalence between $(1)$ and $(2).$ Note that $P$ is necessarily Schur-positive since it is a character.

We proceed to show the equivalence of $(2)$ and $(3)$.
If each summand of $P$ has a total degree which is a multiple of three, then $P(\xi,\xi,\xi)=P(1,1,1)$; this shows that (2) implies (3).
To show the reverse direction, without loss of generality, assume $P\neq 0$ is a positive linear combination of Schur polynomials, none of which has total degree divisible by three.
Let $R(x) = P(x,x,x)$, which is a rational univariate polynomial in $\Q[x]$.
Note that $R$ has all coefficients positive and no term of degree divisible by three.
Now let $R'\in\Q[x]$ be the remainder of division of $R$ by $(x^3-1)$.
Then $R'\neq 0$, has no constant term, and satisfies $R'(1) = R(1)$ and $R'(\xi) = R(\xi)$.
We need to show $P(1,1,1) \neq P(\xi,\xi,\xi)$.
So let us assume the contrary.
Then $R'(\xi) = R'(1)$, so the polynomial $R'(x)-R'(1)$ must be a multiple of the minimal polynomial of $\xi$, namely $x^2+x+1$.
Since $R'$ has degree at most two, we have $R' = c(x^2+x+1)$ for some nonzero constant $c$, but that contradicts that $R'$ has no constant term.
We conclude that $P(1,1,1) \neq P(\xi,\xi,\xi)$, and this completes our proof.
\end{proof}

\begin{example}\label{exmp:byhand}
For simplicity, we denote the unitary representation of $S^1$ with character $\sum_i a_i z^i$ by the list $(i^{a_i}; i \in \Z)$; where we omit the entry of $i$ whenever $a_i=0.$ In particular, in this notation we have $\psi_{a,b,c} = (a,b,c)$.
We now revisit the example at the beginning of the introduction. The computation
$$(-1,0,1)^{\otimes 2} = (-2,-1^2,0^3,1^2,2) = (-2,1^2) \oplus (-2,1^2)^{*} \oplus (0)^{\oplus 3}$$
shows that the representations $\psi_{-1,0,1}$ and $\psi_{-2,1,1}$ can be amalgamated inside $\SU(9)$. In terms of polynomials, this corresponds to $P(x_1,x_2,x_3)=(x_1+x_2+x_3)^2$, $Q(x_1,x_2,x_3) = x_1+x_2+x_3+x_2x_3+x_1x_3+x_1x_2+3x_1x_2x_3$ and the identity
$P(z^{-1},1,z)=Q(z^{-2},z,z).$
Observe that $P=s_{1,1}+s_2$ and $Q=s_1+s_{1,1}+3s_{1,1,1}$.
In particular, both polynomials are Schur positive.
Moreover, $P(1,1,1)=Q(1,1,1)=9$ is the dimension of the representation.
Apart from the this example, which we were able to work out by hand, only few other cases seem suitable for pen and paper calculations.
\end{example}

%Let $P = \sum c_\lambda s_\lambda$ be a positive linear combination of Schur polynomials with coefficients $c_\lambda>0$.  
It follows from the reasoning above that we can formulate Bergman's problem \cite[Question 20]{MR893156} in the first non-trivial case as follows:

\begin{quest}\label{quest:initial}
Given integer vectors $v = (v_1,v_2,v_3)$ and $w = (w_1,w_2,w_3)$ satisfying $v_1+v_2+v_3=w_1+w_2+w_3=0$ and $\gcd(v_1,v_2,v_3) = \gcd(w_1,w_2,w_3) = 1$, can we find Schur positive polynomials in three variables $P$ and $Q$ such that:
\begin{enumerate}
\item $P_v(z) = Q_w(z)$ and
\item $P(\xi,\xi,\xi) \neq P(1,1,1)$ and
\item $Q(\xi,\xi,\xi)\neq Q(1,1,1)$?
\end{enumerate}
\end{quest}

Our experiments suggest that the answer is always positive.

\section{Computations}

Now we recast Question~\ref{quest:initial} as a problem in polyhedral geometry and approach it computationally.
The source code and its output can be found on our \MathRepo page
\begin{equation}\label{eq:link}
  \text{\url{https://mathrepo.mis.mpg.de/CompactAmalgamation/index.html} }.
\end{equation}
To this end we fix $v, w \in \Z^3$ such that $v_1+v_2+v_3=w_1+w_2+w_3=0$ and $\gcd(v_1,v_2,v_3) = \gcd(w_1,w_2,w_3) = 1$.
Choosing an ordering for the Schur polynomials, we then make the problem finite by fixing a number $k$ and considering only the first $k$ Schur polynomials in three variables, denoted by $S_1,\dots,S_k\in\Q[x_1,x_2,x_3]$.
We search for $P = \lambda_1 S_1 + \dots + \lambda_k S_k$ and $Q = \mu_1 S_1 + \dots + \mu_k S_k$, where $\lambda_i, \mu_i$ are non-negative integers.
These polynomials lie in $\Q[x_1,x_2,x_3]$, and their substitutions $P_{v}(z)$ and $Q_{w}(z)$ are univariate Laurent polynomials.
The coefficients of the difference $P_{v}(z)-Q_w(z)$ are integer linear combinations of $\lambda_i$ and $\mu_i$.
Setting these coefficients to zero and letting $\lambda_i \geq 0$ and $\mu_i \geq 0$ defines a polyhedral cone in $\R^{2k}$.
We denote that cone $\cC=\cC_k(v,w)$.
Recall that $\cC$ depends on the chosen ordering of the Schur polynomials.
Throughout we assume that $S_1=s_{1,1,1}$ is the trivial representation.
It plays a special role, as $P = Q = \lambda_1 S_1$, for any $\lambda_1\geq 0$, is a trivial solution to (1) in Question~\ref{quest:initial}.

To find $P$ and $Q$, we consider the integer linear program
\begin{equation}\label{eq:ilp}\tag{ILP$_k$}
  \begin{array}{rl}
    \text{minimize} & c\cdot (\lambda,\mu)\\
    \text{subject to} & (\lambda,\mu) \in \cC_k(v,w) \\
%    & \lambda_2+\dots+\lambda_k+\mu_2+\dots+\mu_k \geq 1 \\ implied by the two subsequent ineqs
    & \sum_{3 \not \; | \tdeg(S_i)} \lambda_i \geq 1 \\
    & \sum_{3 \not \; | \tdeg(S_i)} \mu_i \geq 1 \\
    & \lambda_1,\dots,\lambda_k,\mu_1,\dots,\mu_k\in\N \enspace ,
  \end{array}
\end{equation}
where $c\in\R_{>0}^{2k}$ is some strictly positive linear objective function, to be discussed below.
Let $\cP=\cP_k(v,w)$ be the feasible region of the linear relaxation of \eqref{eq:ilp}.

\begin{rmk}
 Conceptually, one could replace the  weak inequality
 $\sum_{3 \not \; | \tdeg(S_i)} \lambda_i \geq 1 $ by the strict inequality $\sum_{3 \not \; | \tdeg(S_i)} \lambda_i >0$, but the description as an (integer) linear program requires weak inequalities.
\end{rmk}

\begin{prop}
The feasible solutions of \eqref{eq:ilp}, i.e., the lattice points in $\cP$, are in bijection with those nontrivial solutions to Question~\ref{quest:initial} which can be written as a non-negative linear combination of the first $k$ Schur polynomials.
\end{prop}
\begin{proof}
  Containment in the cone $\cC$ is equivalent to the condition (1) in Question~\ref{quest:initial}.
  The two additional constraints correspond to conditions (2) and (3); see Proposition~\ref{prop:injective}.
\end{proof}

\begin{rmk}\label{rmk:choices}
In practice, we make the following choices.
We order the $3$-variate Schur polynomials lexicographically:
a partition $(m+n,n)$, with $m, n\geq 0$, is less than another partition $(m'+n',n')$ if either $m+n < m'+n'$ or $m+n = m'+n'$ and $n < n'$; and the special partition $(1,1,1)$ is defined to be smaller than $(m+n,n)$ for arbitrary $m$ and $n$.
Moreover, we take the objective function $c=(c_i)$ with $c_i=\tdeg S_i$, where $\tdeg$ is the total degree.  So the optimal solutions are minimal with respect to dimension.
\end{rmk}
\noindent We abbreviate $(m)=(m,0)$.
\begin{example}
We consider $v=(-1,0,1)$ and $w=(-2,1,1)$ as in Example~\ref{exmp:byhand}, and we pick $k=4$.
Then the first four Schur polynomials correspond to the partitions $(1,1,1)$, $(1)$, $(1,1)$, and $(2)$.
So we have $S_1=x_1x_2x_3$, $S_2=x_1+x_2+x_3$, $S_3=x_1x_2+x_1x_3+x_2x_3$, and $S_4=x_1^2 + x_1x_2 + x_1x_3 + x_2^2 + x_2x_3 + x_3^2$.
Then
\[
\begin{split}
 P_v(z)-P_w(z) \ = \ (&\lambda_{4} - \mu_{3} - 3\mu_{4})z^2 + (\lambda_{2} + \lambda_{3} + \lambda_{4} - 2\mu_{2})z + \lambda_{1} + \lambda_{2} + \lambda_{3} + 2\lambda_{4} - \mu_{1}\\ &+ (\lambda_{2} + \lambda_{3} + \lambda_{4} - 2\mu_{3} - 2\mu_{4})z^{-1} + (\lambda_{4} - \mu_{2})z^{-2} - \mu_{4}z^{-4} \enspace .
\end{split}
\]
Consequently, the unbounded polyhedron $\cP$ in $\R^8$ is given by six homogeneous equations (from the coefficients of $P_v(z)-P_w(z)$, considered as a Laurent polynomial in $\Q[\lambda_1,\dots,\mu_4][z^\pm]$), the eight nonnegativity constraints and two affine inequalities (from forcing injectivity).
The polyhedron $\cP$ is $3$-dimensional.
Solving the integer linear program \eqref{eq:ilp} yields
\[
\lambda_1=\lambda_2=0\,,\; \lambda_3=\lambda_4=1 \quad \text{and} \quad \mu_1=3\,,\; \mu_2=\mu_3=1\,,\; \mu_4=0
\]
as an optimal solution of objective value $3+6=3+3+3=9$.
This recovers the pair of 9-dimensional representations given by $P=s_{1,1}+s_2$ and $Q=3s_{1,1,1}+s_1+s_{1,1}$ from Example~\ref{exmp:byhand}.
That pair of Schur positive polynomials corresponds to the lattice point marked $0011\, 3110$ in Figure~\ref{fig:byhand}.
Our visualization artificially truncates the feasible region at representation dimension ten.
We see two 9-dimensional solutions and two 10-dimensional ones. The solutions come in pairs since $s_{1,v}=s_{(1,1),v}$ for the special choice of $v=(-1,0,1)$.
The 10-dimensional solutions are obtained from the 9-dimensional solutions by adding a trivial representation.
In this way, the solution from Example \ref{exmp:byhand} explains all four solutions shown here.
\end{example}

\begin{figure}
  \centering
% polymake for joswig
% Wed Apr 12 14:18:09 2023
% pt
\begin{tikzpicture}[z  = {(-0.076cm,-0.9cm)},
                    x  = {(0.95cm,0.06cm)},
                    y  = {(0.29cm,-0.44cm)},
                    scale = 2,
                    color = {lightgray}]

  % DEF COORDINATES
  \coordinate (v0_pt) at (5.5, 0.5, 0);
  \coordinate (v1_pt) at (5.5, 0, 0.5);
  \coordinate (v2_pt) at (0, 0, 1.11111);
  \coordinate (v3_pt) at (0, 0, 0.5);
  \coordinate (v4_pt) at (0, 0.5, 0);
  \coordinate (v5_pt) at (0, 1.11111, 0);

  % VERTEXCOLOR
  \definecolor{vertexcolor_pt}{rgb}{ 1 0 0 }

  % DEF VERTEXSTYLES
  \tikzstyle{vertexstyle_pt} = [] %[circle, scale=0.25pt, fill=vertexcolor_pt,]

  % FACETCOLOR
  \definecolor{facetcolor_pt}{rgb}{ 0.4666666667 0.9254901961 0.6196078431 }

  % EDGECOLOR
  \definecolor{edgecolor_pt}{rgb}{ 0.5 0.5 0.5 }
  \tikzstyle{facetstyle_pt} = [fill=facetcolor_pt, fill opacity=0.1, draw=edgecolor_pt, line width=0.5 pt, line cap=round, line join=round]

  % FACES and EDGES and POINTS in the right order
  \draw[facetstyle_pt] (v0_pt) -- (v1_pt) -- (v3_pt) -- (v4_pt) -- (v0_pt) -- cycle;
  \draw[facetstyle_pt] (v4_pt) -- (v3_pt) -- (v2_pt) -- (v5_pt) -- (v4_pt) -- cycle;
  \draw[facetstyle_pt] (v3_pt) -- (v1_pt) -- (v2_pt) -- (v3_pt) -- cycle;
  \draw[facetstyle_pt] (v0_pt) -- (v4_pt) -- (v5_pt) -- (v0_pt) -- cycle;

  % POINTS
   \node at (v3_pt) [vertexstyle_pt] {};
   \node at (v4_pt) [vertexstyle_pt] {};

  \draw[facetstyle_pt] (v2_pt) -- (v1_pt) -- (v0_pt) -- (v5_pt) -- (v2_pt) -- cycle;

  % POINTS
  \foreach \i in {2,5,1,0} {
    \node at (v\i_pt) [vertexstyle_pt] { };
  }

  % artificial facet (for truncation)
  % DEF COORDINATES
  \coordinate (v0_trunc) at (5.5, 0.5, 0);
  \coordinate (v1_trunc) at (5.5, 0, 0.5);
  \coordinate (v2_trunc) at (0, 0, 1.11111);
  \coordinate (v3_trunc) at (0, 1.11111, 0);

  % VERTEXCOLOR
  \definecolor{vertexcolor_trunc}{rgb}{ 1 0 0 }

  % DEF VERTEXSTYLES
  \tikzstyle{vertexstyle_trunc} = [] % [circle, scale=0.25pt, fill=vertexcolor_trunc,]

  % FACETCOLOR
  \definecolor{facetcolor_trunc}{rgb}{ 1 0 0 }

  % EDGECOLOR
  \definecolor{edgecolor_trunc}{rgb}{ 0 0 0 }
  \tikzstyle{facetstyle_trunc} = [fill=facetcolor_trunc, fill opacity=0.2] %, draw=edgecolor_trunc, line width=1 pt, line cap=round, line join=round]

  % FACES and EDGES and POINTS in the right order
  \draw[facetstyle_trunc] (v2_trunc) -- (v1_trunc) -- (v0_trunc) -- (v3_trunc) -- (v2_trunc) -- cycle;

  % POINTS
  \foreach \i in {2,3,1,0} {
    \node at (v\i_trunc) [vertexstyle_trunc] {};
  }

  % lattice points
  % DEF COORDINATES
  \coordinate (v0_lat_points) at (0, 0, 1);
  \coordinate (v1_lat_points) at (0, 1, 0);
  \coordinate (v2_lat_points) at (1, 0, 1);
  \coordinate (v3_lat_points) at (1, 1, 0);

  % VERTEXCOLOR
  \definecolor{vertexcolor_lat_points}{rgb}{ 1 0 0 }

  % DEF VERTEXSTYLES
  \tikzstyle{vertexstyle_lat_points_0} = [circle, scale=0.25, fill=vertexcolor_lat_points,label={[text=black, below left, align=right]:0011 3110},]
  \tikzstyle{vertexstyle_lat_points_1} = [circle, scale=0.25, fill=vertexcolor_lat_points,label={[text=black, above left, align=right]:0101 3110},]
  \tikzstyle{vertexstyle_lat_points_2} = [circle, scale=0.25, fill=vertexcolor_lat_points,label={[text=black, below right, align=left]:1011 4110},]
  \tikzstyle{vertexstyle_lat_points_3} = [circle, scale=0.25, fill=vertexcolor_lat_points,label={[text=black, above right, align=left]:1101 4110},]

  % FACETCOLOR
  \definecolor{facetcolor_lat_points}{rgb}{ 0 0 1 }

  % EDGECOLOR
  \definecolor{edgecolor_lat_points}{rgb}{ 0 0 0 }
  \tikzstyle{facetstyle_lat_points} = [fill=facetcolor_lat_points, fill opacity=0.5, draw=edgecolor_lat_points, line width=1 pt, line cap=round, line join=round]

  % FACES and EDGES and POINTS in the right order
  \draw[facetstyle_lat_points] (v0_lat_points) -- (v2_lat_points) -- (v3_lat_points) -- (v1_lat_points) -- (v0_lat_points) -- cycle;

  % POINTS
  \foreach \i in {0,2,1,3} {
    \node at (v\i_lat_points) [vertexstyle_lat_points_\i] {};
  }

  % repeat (for fake 3d)
  \draw[facetstyle_pt] (v1_pt) -- (v3_pt);
  
\end{tikzpicture}
    \caption{Four integral points in $\cP_4(v,w)$ for $v=(-1,0,1)$ and $w=(-2,1,1)$.
    Visualized with \polymake \cite{DMV:polymake}; hyperplane for artificial truncation at representation dimension 10 marked red.}
    \label{fig:byhand}
\end{figure}
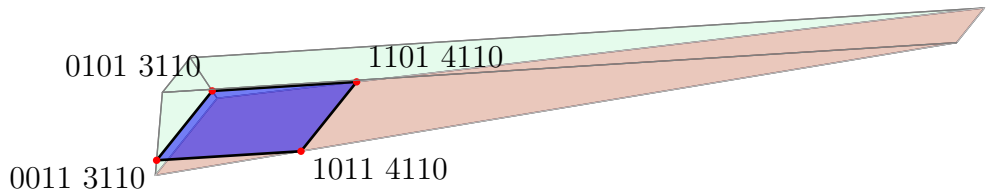

Solving integer linear programs is generally hard, both theoretically and in practice \cite{Schrijver:TOLIP}.
However, our integer linear program \eqref{eq:ilp} has a particularly simple structure, which can be exploited computationally.
\begin{lem}
Let $(\lambda,\mu)\in\Q^{2k}$ be a rational point in $\cP$.
Then there is a positive integer $\ell>0$ such that $(\ell\cdot\lambda,\ell\cdot\mu)$ is a point in $\cP$ which is integral.
\end{lem}
\begin{proof}
  Let $\ell$ be the common denominator of $\lambda_1,\lambda_2,\dots,\mu_k$.
  Then $(\ell\cdot\lambda,\ell\cdot\mu)$ is integral.
  The polyhedron $\cP$ is the intersection of the cone $\cC$ with two additional affine halfspaces.
  Clearly, $(\ell\cdot\lambda,\ell\cdot\mu)$ lies in $\cC$.
  Further, we have $\ell \cdot \sum_{3 \not \; | \tdeg(S_i)} \lambda_i \geq \ell \geq 1$, and similarly for the other inequality.
  Thus the point $(\ell\cdot\lambda,\ell\cdot\mu)$ lies in $\cP\cap\Z^{2k}$.
\end{proof}
As a consequence, the integer linear program \eqref{eq:ilp} is feasible if and only if its linear relaxation is.
The latter condition can be tested much faster.
Consequently, standard complexity bounds in linear optimization entail the following result; see \cites{GroetschelLovaszSchrijver93,Renegar:2001}.
\begin{prop}
Employing the interior point method, deciding the feasibility of the integer linear program \eqref{eq:ilp} takes polynomial time in the five parameters $k$, $\log |v_1|$, $\log |v_2|$, $\log |w_1|$, and $\log |w_2|$.
\end{prop}
Recall the condition $v_1+v_2+v_3=0=w_1+w_2+w_3$, whence $v_3$ and $w_3$ are not mentioned.
Now we can summarize how to address Question~\ref{quest:initial} computationally.
First we pick some integer $k$.
Then we decide the feasibility of \eqref{eq:ilp} by solving the linear relaxation.
If this is feasible we use a bisection to find the minimal $k'$ such that (ILP$_{k'}$) is feasible.
If it is infeasible we try $2k$ and repeat.
Of course, this procedure does not terminate if no solution exists.
Yet that did not occur so far.

There are many implementations of algorithms for linear and integer optimization available, both open source and commercial.
Yet the majority employs floating-point arithmetic, which may lead to errors, which in turn makes these software systems less suited for obtaining mathematical results.
For this reason we use \scip, which implements the simplex method in exact rational arithmetic \cite{scip}.
Setting up the (integer) linear program \eqref{eq:ilp} is done in \OSCAR, which provides partitions, Schur polynomials and the necessary commutative algebra \cites{OSCAR-book,OSCAR}.
\OSCAR also inherits the full functionality of \polymake \cite{DMV:polymake}, which includes exact rational integer linear programming.
While \scip is much faster at integer linear programming, that implementation is based on floating-point arithmetic.

\begin{table}[th]
\caption{Minimal $k$ for which $\cP_k(v,w)$ is feasible, where $v=(1,v_2,-1-v_2)$ and $w=(1,w_2,-1-w_2)$.
Empty fields on the upper right are beyond our current reach computationally.}
\label{table: min k}
\renewcommand{\arraystretch}{0.9}
\begin{tabular*}{.8\linewidth}{@{\extracolsep{\fill}}rrrrrrrrrrrr@{}}
\toprule
$w_2\backslash v_2$ & \multicolumn{1}{c}{0} & \multicolumn{1}{c}{1} & \multicolumn{1}{c}{2} & \multicolumn{1}{c}{3} & \multicolumn{1}{c}{4} & \multicolumn{1}{c}{5} & \multicolumn{1}{c}{6} & \multicolumn{1}{c}{7} & \multicolumn{1}{c}{8} & \multicolumn{1}{c}{9} & \multicolumn{1}{c}{10} \\
\midrule
0 &  & 4 & 16 & 191 & 601 & 1541 &  &  &  &  &  \\
1 &  &  & 13 & 33 & 106 & 336 & 686 & 1254 & 2187 &  &  \\
2 &  &  &  & 21 & 50 & 125 & 305 & 586 & 1006 & 1574 &  \\
3 &  &  &  &  & 28 & 66 & 170 & 292 & 535 & 820 & 1283 \\
4 &  &  &  &  &  & 44 & 86 & 174 & 307 & 463 & 824 \\
5 &  &  &  &  &  &  & 61 & 120 & 238 & 377 & 525 \\
6 &  &  &  &  &  &  &  & 87 & 171 & 275 & 430 \\
7 &  &  &  &  &  &  &  &  & 115 & 245 & 333 \\
8 &  &  &  &  &  &  &  &  &  & 145 & 291 \\
9 &  &  &  &  &  &  &  &  &  &  & 171 \\
\bottomrule
\end{tabular*}
\end{table}

\subsection*{Feasibility}
For our first experiment, we consider pairs of vectors $v=(v_1,v_2,-v_1-v_2)$ and $w=(w_1,w_2,-w_1-w_2)$ such that $v_1=w_1=1$.
Such a pair $(v,w)$ is determined by the pair $(v_2,w_2)$ of integers.
In Table \ref{table: min k}, we give the minimal values of $k$ for which $\cP_k(v,w)$ is feasible, which we compute by solving the linear relaxation of \eqref{eq:ilp}.
As pointed out in Remark~\ref{rmk:choices}, the parameter $k$ refers to the lexicographic ordering of the Schur polynomials.
That ordering does affect the value of $k$.
That is to say, replacing the pure lexicographic ordering by, e.g., the graded lexicographic ordering may lead to a lower value of $k$.
It is unclear whether one ordering is better than another.

\subsection*{Representation dimensions}
For our second experiment we actually solve the integer linear program \eqref{eq:ilp}.
We take the objective function $$c = \bigl(S_1(1,1,1),\dots,S_k(1,1,1)\bigr)$$ to be the dimension of the representation corresponding to $P$; see Remark~\ref{rmk:choices}.
Table~\ref{table: dim} records pairs of vectors, the minimal value of $k$ such that $\cP_k(v,w)$ is feasible and the optimal value of the integer linear program, i.e., the smallest dimension achieved by solutions using only the first $k$ Schur polynomials.
Each row of that table corresponds to one entry in Table~\ref{table: min k}.
The explicit Schur positive symmetric polynomials whose dimensions are recorded in Column 4 of Table~\ref{table: dim} can be found on our \MathRepo page~(\ref{eq:link}) alongside the source code.

\begin{table}[th]
\caption{Minimal $k$ for which $\cP_k(v,w)$ is feasible and the dimension of the representation that corresponds to an optimal integral solution}
\label{table: dim}
   \begin{tabular*}{.75\linewidth}{@{\extracolsep{\fill}}ccrr@{}}
   \toprule
    $v$ & $w$ & $k$ & Dimension \\
    \midrule
    $(1,0,-1)$ & $(1,1,-2)$ & 4 & 9 \\
    $(1,0,-1)$ & $(1,2,-3)$ & 16  & 21 \\
    $(1,1,-2)$ & $(1,2,-3)$ & 13 & 63 \\
    $(1,1,-2)$ & $(1,3,-4)$ & 33 & 834 \\
    $(1,1,-2)$ & $(1,4,-5)$ & 106 & 3216 \\
    $(1,2,-3)$ & $(1,3,-4)$ & 21 & 255 \\
    $(1,2,-3)$ & $(1,5,-6)$ & 125 & 13561 \\
    $(1,3,-4)$ & $(1,4,-5)$ & 28 & 454 \\
    $(1,3,-4)$ & $(1,5,-6)$ & 66 & 6852 \\
    $(1,4,-5)$ & $(1,5,-6)$ & 44 & 1526 \\
    $(1,4,-5)$ & $(1,6,-7)$ & 86 & 83113 \\
    $(1,5,-6)$ & $(1,6,-7)$ & 61 & 14972 \\
    $(1,5,-6)$ & $(1,7,-8)$ & 120 & 316170 \\
    $(1,6,-7)$ & $(1,7,-8)$ & 87 & 128624 \\
    $(1,7,-8)$ & $(1,8,-9)$ & 115 & 108468 \\
    \bottomrule
    \end{tabular*}
 \end{table}

Note that the dimensions recorded in Table \ref{table: dim} might not be minimal among all solutions since they use only the first $k$ Schur polynomials; allowing the use of more Schur polynomials can potential provide a solution with smaller dimension.

\subsection*{Running times}
We briefly comment on the computation time of Table~\ref{table: min k} and Table~\ref{table: dim}.
Computing all entries in Table~\ref{table: min k} took in total approximately 400,000 seconds (4.6 days).
Optimal solutions in Table~\ref{table: dim} are computed in \scip, via floating-point arithmetic, and then verified in \OSCAR, via exact arithmetic.
Verification is fast and succeeded in all our cases.
The longest computation was for the pair $(1,5,-6)$ and $(1,7,-8)$, which took 312 seconds in \scip.
Computations for pairs with $k > 125$ did not terminate within a day.
  
All computations were done on the computer server Hydra at the MPI MiS, with the following system specifics: 4x16-core Intel Xeon E7-8867 v3 CPU (3300 MHz) on Debian GNU/Linux 5.10.149-2 (2022-10-21) x86\_64.

\begin{rmk}
    In principle, the optimal (rational) solutions to the linear programming relaxations leading to Table~\ref{table: min k} yield an upper bound on the smallest dimension of a representation of the amalgamation problem Question~\ref{quest:initial}. However, these numbers are excessively large. For example, for $v=(1,9,-10)$ and $w=(1,10,-11)$ the bound we obtain is $2382041666750207$. This is one of the smaller ones. Therefore, it is not desirable to provide a complete table here. However, using the \texttt{Jupyter} notebook available on the \MathRepo page (\ref{eq:link}) the interested reader can compute some of these numbers by themselves.
\end{rmk}

\section{A relaxed problem}
In this section, we consider the following relaxed problem by dropping the Schur positivity condition and disregarding the case $(-1,0,1)$.
Recall that the Schur polynomials form a basis of the space of all symmetric polynomials; see \cite[§A.1]{MR1153249}.

\begin{quest}\label{quest:initial2}
Given $v = (v_1,v_2,v_3)$ such that $v_1+v_2+v_3=0$, $v_1v_2v_3\neq0$ %\todo{That condition did not occur before.  Explain?} 
and $\gcd(v_1,v_2,v_3) = 1$. For which Laurent polynomials $F\in\Q[z^\pm]$ can we find a symmetric polynomial in three variables $P$ that $F=P_v(z)$?
\end{quest}

\begin{rmk}
 We pose the additional condition $v_1v_2v_3\neq0$, which excludes the case $(v_1,v_2,v_3)=(-1,0,1)$, because our argument does not apply to that case, see Remark~\ref{rem:exludecase}.
\end{rmk}

From now on fix a triplet $v = (v_1,v_2,v_3)$ such that $v_1+v_2+v_3=0$, $v_1v_2v_3\neq0$ and $\gcd(v_1,v_2,v_3) = 1$.
Since every symmetric polynomial in three variables can be written as a polynomial in the first three elementary symmetric polynomials $$e_1=x_1+x_2+x_3 ,\ e_2=x_1x_2+x_1x_3+x_2x_3,\ e_3=x_1x_2x_3$$ in $\Q[x_1,x_2,x_3]$, and because $v_1+v_2+v_3=0$ implies that $(e_3)_v(z)=1$, answering Question \ref{quest:initial2} amounts to characterizing the $\Q$-subalgebra $A(v)$ of the Laurent polynomial ring $\Q[z^\pm]$ generated by $$F_1:=(e_1)_v(z)=z^{v_1}+z^{v_2}+z^{v_3} \quad \text{and} \quad F_2:=(e_2)_v(z)=z^{v_1+v_2}+z^{v_1+v_3}+z^{v_2+v_3}.$$
Since we have $F_1'(1)=F_2'(1)=0$, the product rule implies that $F'(1)=0$ holds for all $F\in A(v)$.
This shows that $A(v)$ is a proper subalgebra of $\Q[z^\pm]$.
As the next example shows, this is in general not the only constraint.

\begin{example}\label{ex:3droot}
  Consider the case $v=(1,1,-2)$, and let $\xi\in\C$ be a primitive third root of unity.
  We have $F_1'=2\cdot(1-z^{-3})$ and $F_2'=2z\cdot(1-z^{-3})$ and this shows $F_1'(\xi)=F_2'(\xi)=0$.
  Again this shows that $F'(\xi)=0$ for all $F\in A(1,1,-2)$. One can prove that these are all constraints in this case:
  \begin{equation*}
    A(1,1,-2)=\{F\in\Q[z^\pm]\mid F'(1)=F'(\xi)=F'(\xi^2)=0 \}.
  \end{equation*}
\end{example}

Our main contribution in this section is the following rather technical result which says that $A(v)$ can, in general, be characterized by conditions similar as in Example \ref{ex:3droot}.

\begin{thm}\label{thm:main}
 There is a product $\Phi\in\Q[z]$ of cyclotomic polynomials with $\Phi(1)\neq0$ and a subalgebra $C$ of $\Q[z^\pm]/(\Phi)$ such that for $F\in\Q[z^\pm]$ the following are equivalent:%\todo{$Q$ is not a good symbol here, as this was used for the pair $P,Q$ before}
 \begin{enumerate}
  \item There is a symmetric polynomial $P$ in three variables with rational coefficients such that $F=P_v(z)$.
  \item We have $F'(1)=0$, and the residue class of $F$ modulo $\Phi$ is in $C$.
 \end{enumerate}
\end{thm}

\begin{rmk}
 The subalgebra $C$ of $\Q[z^\pm]/(\Phi)$ in Theorem \ref{thm:main} is the one generated by the residue classes of $F_1$ and $F_2$. Since $\Q[z^\pm]/(\Phi)$ is a finite dimensional $\Q$-vector space, this can be explicitly calculated once knowing $\Phi$.
\end{rmk}

Before we will give a proof of Theorem \ref{thm:main} we point out some consequences that are less technical.
    
\begin{cor}\label{cor:vanishingcrit}
 There are finitely many roots of unity $\zeta_1,\ldots,\zeta_r\in\C\setminus\{1\}$ and natural numbers $a_1,\ldots,a_r$ such that every $F\in\Q[z^\pm]$ with $F'(1)=0$ which vanishes at $\zeta_i$ with multiplicity at least $a_i$ for $i=1,\ldots,r$ can be expressed as $F=P_v(z)$ for some symmetric polynomial $P$ in three variables.
\end{cor}

\begin{proof}
 Let $\Phi\in\Q[z]$ the polynomial from Theorem \ref{thm:main} and let%\todo{after adjusting $Q$, the symbols for $Q_i$ should be changed likewise}
 \begin{equation*}
  \Phi=\Phi_1^{a_1}\cdots \Phi_s^{a_s}
 \end{equation*}
 where the $\Phi_i$ are pairwise coprime cyclotomic polynomials. If $F\in\Q[z^\pm]$ vanishes at the zeros of each $\Phi_i$ with multiplicity at least $a_i$, then $F$ is divisible by $\Phi$. Thus the residue class of $F$ modulo $\Phi$ is zero and hence contained in every subalgebra of $\Q[z^\pm]/(\Phi)$.
\end{proof}

\begin{cor}\label{cor:concrete}
 There are natural numbers $a_0,b_0>0$ such that for all $a\geq a_0$, all $b$ divisible by $b_0$ we have
 \begin{equation}\label{eq:F_nm}
   F_{a,b}=(1+z+\cdots+z^{b-1})^a\cdot (1+z^{-1}+\cdots+z^{-(b-1)})^a\in A(v).
 \end{equation}
\end{cor}

\begin{proof}
 Let  $\zeta_1,\ldots,\zeta_r$ and $a_1,\ldots,a_r$ as in Corollary \ref{cor:vanishingcrit}. 
 Let $b_0$ such that $\zeta_i^{b_0}=1$ for all $i=1,\ldots,r$ and $a\geq\frac{1}{2}\max_{i=1}^r (a_i)$. Then for all $a\geq a_0$ and all $b$ divisible by $b_0$ the Laurent polynomial $F_{a,b}$ vanishes at $\zeta_i$ with multiplicity at least $a_i$ for $i=1,\ldots,r$. A straight-forward calculation further shows that $F'_{a,b}(1)=0$.
\end{proof}

\begin{cor}
 Consider finitely many triplets
 \begin{equation*}
  t_1,\ldots,t_r\in\{(\alpha,\beta,\gamma)\in\Z^3\mid \alpha+\beta+\gamma)=0, \alpha\beta\gamma\neq0\,\textrm{ and }\gcd(\alpha,\beta,\gamma) = 1\}.
 \end{equation*}
 Then there are natural numbers $a,b$ such that $F_{a,b}\in\bigcap_{i=1}^r A(t_i)$.
\end{cor}

\begin{proof}
 For each $i\in\{1,\ldots,r\}$ we obtain $a_0$ and $b_0$ as in Corollary \ref{cor:concrete}. We can choose $a$ as the maximum of all such $a_0$ and $b$ as the product of all such $b_0$.
\end{proof}

In fact, we conjecture that the Laurent polynomials $F_{a,b}$ in Equation \eqref{eq:F_nm} can even be realized as positive rational linear combinations of Schur polynomials.

\begin{conj}\label{con:speccon}
 There are natural numbers $a_0,b_0>0$ such that for all $a\geq a_0$, all $b$ divisible by $b_0$ there is $N\in\N$ and a Schur positive symmetric polynomial $P$ in three variables such that
 \begin{equation*}
  N\cdot F_{a,b}= P_v.
 \end{equation*}
\end{conj}

In order to amalgamate two representations given by tuples $v$ and $w$ let $a_0,b_0,N$ and $a_0',b_0',N'$ the natural numbers from the previous conjecture for $v$ and $w$ respectively. Then, if Conjecture \ref{con:speccon} is true, letting $a=\max(a_0,a_0')$, $n=\lcm(b_0,b_0')$ and $M=\lcm(N,N')$, we have
\begin{equation*}
    P_v=M\cdot F_{a,b}=Q_{w}
\end{equation*}
for Schur positive symmetric polynomials $P$ and $Q$.%\todo{It is worth considering to change $m$ and $n$ (and their ilk) because $m$ and $n$ were used for the lengths of the partitions before.}

\begin{rmk}\label{rmk:testconj}
 In the case $v = (1,1,-2)$ our computational experiments suggest that Conjecture~\ref{con:speccon} is true for $a_0=1$ and $b_0=3$.
  %We test Conjecture~\ref{con:speccon} computationally for $v = (1,1,-2)$ for various values of $n$ and $m$. By Example~\ref{ex:11}, we set $n_0 = 3$. We set $m_0=1$. 
  We found $N$ and Schur-positive symmetric polynomials $P_v$ that satisfy $N\cdot F_{a,b} = P_v$ for various pairs of $a$ and $b$. We record the values of $N$ and the dimensions of $P_v$ in Table~\ref{table: conjecture}.

\begin{table}[th]
  \caption{Experimental data on Conjecture \ref{con:speccon} with $v = (1,1,-2)$}
  \label{table: conjecture}
   \begin{tabular*}{.3\linewidth}{@{\extracolsep{\fill}}ccrr@{}}
   \toprule
    $b$ & $a$ & $N$ & $\dim P_v$ \\
    \midrule
    3 & 1 & 1 & 9 \\
    6 & 1 & 2 & 72 \\
    9 & 1 & 3 & 243 \\
    12 & 1 & 4 & 576 \\
    15 & 1 & 5 & 1125 \\
    \bottomrule
    \end{tabular*}
    \quad
    \begin{tabular*}{.3\linewidth}{@{\extracolsep{\fill}}ccrr@{}}
   \toprule
    $b$ & $a$ & $N$ & $\dim P_v$ \\
    \midrule
    3 & 2 & 1 & 81 \\
    6 & 2 & 1 & 1296 \\
    9 & 2 & 1 & 6561 \\
    12 & 2 & 1 & 20736 \\
    15 & 2 & 1 & 50625 \\
    \bottomrule
    \end{tabular*}
    \quad
    \begin{tabular*}{.3\linewidth}{@{\extracolsep{\fill}}ccrr@{}}
   \toprule
    $b$ & $a$ & $N$ & $\dim P_v$ \\
    \midrule
    3 & 3 & 1 & 729 \\
    6 & 3 & 2 & 93312 \\
    9 & 3 & 1 & 531441 \\
    12 & 3 & 2 & 5971968 \\
    \\
    \bottomrule
    \end{tabular*}
 \end{table}

The computations are similar to what we perform in the previous section. Given $a,b$ and $N$, we obtain a Laurent polynomial $N\cdot F_{a,b}$. The degree of this polynomial gives an upper bound on the degrees of the Schur polynomials that can appear in $P_v$. We then take all available Schur polynomials and solve an integral linear program like before. The source code and explicit polynomials $P_v$ can be found on our \MathRepo page~(\ref{eq:link}).
\end{rmk}

\subsection*{Proof of Theorem \ref{thm:main}}\label{sec:proofalgeom}
Our proof involves some algebraic geometry; see the textbooks by Hartshorne \cite{Hart77} and Harris \cite{Ha95}.
We consider the polynomial map
\begin{equation}\label{eq:curve}
  f\colon \C^*\to \C^2,\, z\mapsto \bigl(F_1(z),F_2(z)\bigr)=(z^{v_1}+z^{v_2}+z^{v_3},z^{v_1+v_2}+z^{v_1+v_3}+z^{v_2+v_3}).
\end{equation}

We first study where this map fails to be injective.

\begin{lem}\label{lem:fibone}
 We have $f^{-1}(f(1))=\{1\}$.
\end{lem}

\begin{proof}
  Let $x\in\C^*$ such that $f(x)=f(1)$.
  This implies
  \begin{equation*}
    (t-x^{v_1})(t-x^{v_2})(t-x^{v_3})=t^3-F_1(x)t^2+F_2(x)t-1=t^3-F_1(1)t^2+F_2(1)t-1=(t-1)^3,
  \end{equation*}
  which entails $x^{v_1}=x^{v_2}=x^{v_3}=1$. Since $\gcd(v_1,v_2,v_3)=1$, we get $x=1$. 
\end{proof}

For a complex number $x\in\C$, let $|x|=\sqrt{x\cdot\overline{x}}$ be its norm.

\begin{lem}\label{lem:realfiber}
 For $|x|\neq 1$ we have $|f^{-1}(f(x))|=1$.
\end{lem}

\begin{proof}
 Let $y\in\C^*$ such that $f(y)=f(x)$. This implies that the zeros of the polynomial $$(t-y^{v_1})(t-y^{v_2})(t-y^{v_3})$$ are the three complex numbers $x^{v_1}$, $x^{v_2}$ and $x^{v_3}$. If $|x|>1$, then $|y|>1$ as well. Indeed, if two of the three integers $v_1,v_2,v_3$ are positive, then two of the three real numbers $|x|^{v_1}$, $|x|^{v_2}$ and $|x|^{v_3}$ are larger than one and thus the same must hold for the real numbers $|y|^{v_1}$, $|y|^{v_2}$ and $|y|^{v_3}$. Otherwise two of the three integers $v_1,v_2,v_3$ are negative, and a similar argument applies. Since for every real $t>1$ the map $d\mapsto t^d$ is strictly increasing, we must have $y^{v_1}=x^{v_1}$, $y^{v_2}=x^{v_2}$ and $y^{v_3}=x^{v_3}$. 
 This implies $f(\frac{y}{x})=f(1)$, and hence Lemma \ref{lem:fibone} yields  $y=x$.
 The case $|x|<1$ is analogous.
 \end{proof}

\begin{rmk}\label{rem:exludecase}
  The statement of Lemma~\ref{lem:realfiber} is not true in the case $(v_1,v_2,v_3)=(-1,0,1)$. Indeed, in this case the preimage of $f(x)$ under the map
  \begin{equation}
  f\colon\C^*\to \C^2,\, z\mapsto =(z^{v_1}+z^{v_2}+z^{v_3},z^{v_1+v_2}+z^{v_1+v_3}+z^{v_2+v_3})=(z^{-1}+1+z,z^{-1}+1+z)
\end{equation}
has two elements for all $x\in\mathbb{C}^*\setminus\{-1,1\}$. This is why we have excluded this case.
\end{rmk}

We denote by $B(v)$ the $\C$-subalgebra of $\C[z^\pm]$ generated by $F_1$ and $F_2$. Note that $B(v)=A(v)\otimes_{\mathbb{Q}}\mathbb{C}$ and $A(v)=B(v)\cap\mathbb{Q}[z^\pm]$. %\todo{Maybe add a word on the relationship between $A(v)$ and $B(v)$.}

\begin{lem}\label{lem:finite}
The ring extension $B(v)\subset\C[z^\pm]$ is \emph{finite}, i.e., $\C[z^\pm]$ is finitely generated as a $B(v)$-module.
\end{lem}

\begin{proof}
 For $t\in\{z^{v_1},z^{v_2},z^{v_3}\}$ we have
 \begin{equation*}
     t^3-F_1t^2+F_2t-1=(t-z^{v_1})(t-z^{v_2})(t-z^{v_3})=0.
 \end{equation*}
 This implies that $t^k$, for $k\in\N$, is contained in the $B(v)$-module that is generated by $1,t,t^2$. Thus $\C[z^{v_1},z^{v_2},z^{v_3}]$ is equal to the $B(v)$-module that is generated by
 \begin{equation*}
     \{z^{av_1}z^{bv_2}z^{cv_3}\mid 0\leq a,b,c\leq2\}.
 \end{equation*}
 Now it remains to show that $\C[z^\pm]=\C[z^{v_1},z^{v_2},z^{v_3}]$.
 The inclusion ``$\supset$'' is clear. Since $\gcd(v_1,v_2,v_3)=1$, there are integers $a,b,c$ such that
 \begin{equation*}
  av_1+bv_2+cv_3=1.
 \end{equation*}
 Since $v_1+v_2+v_3=0$ we also have
 \begin{equation*}
  (a+m)v_1+(b+m)v_2+(c+m)v_3=1
 \end{equation*}
for every $m\in\Z$. In particular, we can find natural numbers $a',b',c'$ such that
\begin{equation*}
  a'v_1+b'v_2+c'v_3=1
 \end{equation*}
 meaning that $z=(z^{v_1})^{a'}(z^{v_2})^{b'}(z^{v_3})^{c'}\in \C[z^{v_1},z^{v_2},z^{v_3}]$.
 Analogously, it can be proved that $z^{-1}\in \C[z^{v_1},z^{v_2},z^{v_3}]$.
\end{proof}

The $\C$-algebra $B(v)$ is the coordinate ring of the algebraic curve $X\subset\C^2$ cut out by the elements of the kernel of the map
\begin{equation*}
  \C[x,y]\to B(v),\, P\mapsto P(F_1,F_2);
\end{equation*}
we denote the quotient field of $B(v)$ by $K$.
Lemma \ref{lem:finite} implies that $X$ is, in fact, the image of $f$ because finite morphisms are closed \cite[Exc.~II.4.1]{Hart77}.

\begin{prop}\label{prop:mostlyinj}
  There are only finitely many $x\in\C^*$ such that $|f^{-1}(f(x))|>1$.
  All of them are roots of unity.
  Moreover, we have $K=\C(z)$, the rational function field.
\end{prop}

\begin{proof}
  For every $x\in\C^*$ the fiber $f^{-1}(f(x))$ is Zariski closed in $\C^*$.
  The Zariski closed subsets of $\C^*$ are either finite or all of $\C^*$.
  Therefore, since $f$ is not constant, every fiber $f^{-1}(f(x))$ is finite.
  By \cite[Proposition 7.16]{Ha95} the field extension $\C(z)/K$ %\todo{Should this be $K/\Q(z)$?  And, similarly, $[K:\Q(z)]$ below?} 
  is finite and there is a nonempty Zariski open subset $U\subset \C^*$ such that $|f^{-1}(f(x))|=[\C(z):K]$ for all $x\in U$.
  This implies that $|f^{-1}(f(x))|=[\C(z):K]$ is true for all but finitely many $x\in\C^*$.
  Lemma \ref{lem:realfiber} thus shows that $[\C(z):K]=1$ and that each of the finitely many $x\in\C^*$ with $|f^{-1}(f(x))|>1$ must satisfy $|x|=1$.
  Moreover, since $f$ is defined over $\Q$, all such $x$ are algebraic numbers and therefore roots of unity.
\end{proof}

The situation is very similar for ramification points of $f$.

\begin{prop}\label{prop:rami}
 If $x\in\C^*$ is not a root of unity, then $f$ is unramified at $x$.
\end{prop}

\begin{proof}
 Since every power sum in three variables can be written as a polynomial in the first three elementary symmetric polynomials, there is for every $n\in\N$ a polynomial map
 $\varphi_n\colon \C^2\to\C^n$
 such that 
 \begin{equation*}
 \varphi_n(f(x))=(x^{v_1}+x^{v_2}+x^{v_3},    x^{2v_1}+x^{2v_2}+x^{2v_3},\ldots,x^{nv_1}+x^{nv_2}+x^{nv_3}).
 \end{equation*}
 If $f$ is ramified at $x\in\C^*$, then $\varphi_n\circ f$ is also ramified at $x$ for every $n\in\N$. Thus
 \begin{equation*}
    kv_1x^{kv_1-1}+kv_2x^{kv_2-1}+kv_3x^{kv_3-1}=0 
 \end{equation*}
 for all $k\in\N$. As $x$ and $k$ are nonzero, this implies that
 \begin{equation*}
    v_1x^{kv_1}+v_2x^{kv_2}+v_3x^{kv_3}=0
 \end{equation*}
 for all $k\in\N$. This means that $x^k$ is a zero of the nonconstant polynomial 
 \begin{equation*}
  v_1z^{v_1}+v_2z^{v_2}+v_3z^{v_3}\in\Q[z]   
 \end{equation*}
 for all $k\in\N$. Hence the set
 \begin{equation*}
     \{x^k\mid k\in\N\}
 \end{equation*}
 is finite which implies that $x$ is a root of unity.
\end{proof}

\begin{cor}\label{cor:noniso}
 There is a finite set $S\subset\C^*$ of roots of unity such that 
 \begin{equation*}
  \C^*\setminus S\to X\setminus f(S),\, x\mapsto f(x)   
 \end{equation*}
 is an isomorphism.
\end{cor}
 
 \begin{proof}
 This follows from Proposition \ref{prop:mostlyinj}, Proposition \ref{prop:rami} and Lemma \ref{lem:finite} by \cite[Thm.~14.9]{Ha95}.
 \end{proof}

\begin{rmk}
 The smallest set $S\subset\C^*$, such that \begin{equation*}
  \C^*\setminus S\to X\setminus f(S),\, x\mapsto f(x)   
 \end{equation*}
 is an isomorphism, is the preimage of the singular locus of the curve $X$ under $f$.
\end{rmk}

\begin{example}\label{ex:11}
Let $v=(1,1,-2)$. Then $X$ is the zero set of the bivariate quartic polynomial
\begin{equation*}
  x_1^2 x_2^2-4 x_1^3-4 x_2^3+18 x_1 x_2-27
\end{equation*}
in $\C^2$.
The three points $f(1)=(3,3)$, $f(\xi)=(3\xi,3\xi^2)$ and $f(\xi^2)=(3\xi^2,3\xi)$ form the singular locus.
In particular, its preimage under $f$ is the set of third roots of unity.
See Figure~\ref{fig: X for v=11}.
Note that this motivates the choice $b_0=3$ in Remark~\ref{rmk:testconj}.
\begin{figure}[h]
    \centering
    \includegraphics[width=8cm]{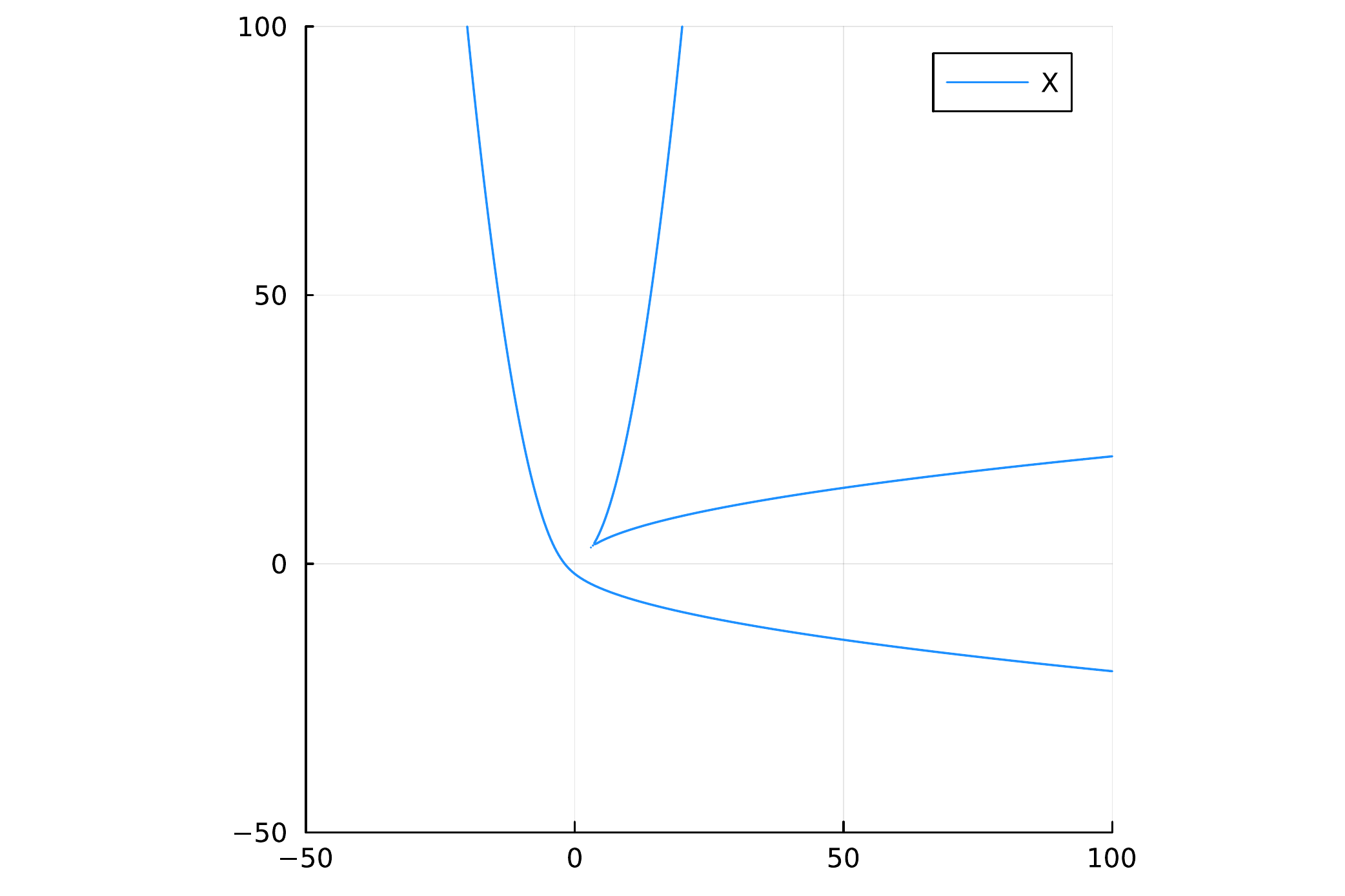}
    \caption{The real locus of $X$ for $v = (1,1,-2)$ plotted in the plane.}
    \label{fig: X for v=11}
\end{figure}
\end{example}

\begin{example}\label{ex:12}
For $v=(1,2,-3)$ one computes that $X$ is cut out by 
\begin{equation*}
  \begin{split}
    x_1^3 x_2^3&-x_1^5-3 x_1^4 x_2-3 x_1 x_2^4-x_2^5-x_1^4+5 x_1^3 x_2+10 x_1^2 x_2^2+5 x_1 x_2^3- x_2^4 \\ &+x_1^3  -x_1^2 x_2-x_1 x_2^2+x_2^3-7 x_1^2-13 x_1 x_2-7 x_2^2.
  \end{split}
\end{equation*}
\normalsize{The preimage of its singular locus under $f$ is the set of all third roots of unity along with the set of primitive seventh and eighth roots of unity.}
\end{example}

Recall that the \emph{conductor} of the ring extension $B(v)\subset\C[z^\pm]$ is defined as 
\begin{equation*}
  I=\{a\in B(v)\mid a\cdot \C[z^\pm]\subset B(v)\} 
\end{equation*}
which is an ideal in both $B(v)$ and $\C[z^\pm]$. The zero set of $I$ is contained in the locus where $f$ fails to be an isomorphism \cite[p.~316]{Bour72Comm}. Thus by Corollary \ref{cor:noniso} and because $\C[z^\pm]$ is a principal ideal domain, it follows that $I$ is generated by a Laurent polynomial all of whose zeros are roots of unity. Since $f$ is defined over $\Q$, this Laurent polynomial is also defined over $\Q$. Therefore, we can write the generator of $I$ as $(z-1)^m\cdot \Phi$ where $\Phi$ is a product of cyclotomic polynomials with $\Phi(1)\neq0$.

\begin{lem}\label{lem:cusp}
 We have $m=2$.
\end{lem}

\begin{proof}
 Recall from \eqref{eq:curve} that $X$ is the image of the map
 \begin{equation*}
 f\colon \C^*\to \C^2,\, z\mapsto \bigl(F_1(z),F_2(z)\bigr)=(z^{v_1}+z^{v_2}+z^{v_3},z^{v_1+v_2}+z^{v_1+v_3}+z^{v_2+v_3}).
\end{equation*}
 We have $F_1'(1)=F_2'(1)=0$ and $F_1''(1)=F_2''(1)=v_1^2+v_2^2+v_3^2\neq0$. Together with Lemma~\ref{lem:fibone} this implies that $X$ has an ordinary cusp at the image of $1$. The coordinate ring of $X$ is $B(v)$. Proposition~\ref{prop:mostlyinj} and Lemma~\ref{lem:finite} imply that $\C[z^\pm]$ is the integral closure of $B(v)$. In this situation the conductor has been computed in \cite[Proposition~1]{fultonblowup}.
\end{proof}

As we are actually interested in polynomials with rational coefficients, we consider $A(v)$ the $\Q$-subalgebra of $\Q[z^\pm]$ generated by $F_1$ and $F_2$ and we let $I'$ the ideal of $\Q[z^\pm]$ generated by $(z-1)^m\cdot \Phi$.

\begin{cor}
 The ideal $I'$ is the conductor of the ring extension $A(v)\subset\Q[z^\pm]$:
 \begin{equation*}
     I'=\{a\in A(v)\mid a\cdot \Q[z^\pm]\subset A(v)\}.
 \end{equation*}
\end{cor}

\begin{proof}
 This follows from $I'=I\cap\Q[z^\pm]$ and $A(v)=B(v)\cap\Q[z^\pm]$.
\end{proof}

Let $C_0\subset \Q[z^\pm]/I'$ the image of $A(v)$ modulo $I'$.

\begin{lem}\label{lem:subalg}
Let $g\in\Q[z^\pm]$. Then we have $g\in A(v)$ if and only if the residue class of $g$ modulo $I'$ is in $C_0$.
\end{lem}

\begin{proof}
One direction is trivial so let us assume that the residue class of $g$ modulo $I'$ is in $C_0$. Thus there is a $h_1\in A(v)$ and $h_2\in I'$ such that $g=h_1+h_2$. Thus $g\in A(v)$ because $I'\subset A(v)$.
\end{proof}
By the Chinese remainder theorem we can naturally identify
\begin{equation*}
    \Q[z^\pm]/I'=(\Q[z^\pm]/(z-1)^2)\times (\Q[z^\pm]/\Phi).
\end{equation*}

\begin{lem}\label{lem:product}
 We have $C_0=\Q\times C$, where
 \begin{equation*}
     C=\{g\in \Q[z^\pm]/\Phi\mid \exists h\in \Q[z^\pm]/(z-1)^2)\colon (h,g)\in C_0\}.
 \end{equation*}
\end{lem}

\begin{proof}
Since $F_i(1)=3$ and the derivative of $F_i$ vanishes at $1$, we have 
\begin{equation*}
 F_i\equiv 3\mod(z-1)^2   
\end{equation*}
for $i=1,2$. This shows the inclusion "$\subset$". For the reverse inclusion we observe that, by Lemma \ref{lem:fibone}, there is a polynomial $G\in\Q[x_1,x_2]$ that vanishes on $f(\zeta)$ for all zeros $\zeta$ of $\Phi$ but $G(f(1))=1$. The residue class modulo $I'$ of a large enough power of $G(F_1,F_2)\in A(v)\subset\Q[z^\pm]$ is then $(1,0)$. In particular, this shows that $(1,0)\in C_0$ and proves the claim.
\end{proof} 

\begin{proof}[Proof of Theorem \ref{thm:main}]
Let $F\in\Q[z^\pm]$. Then by Lemma \ref{lem:subalg} and Lemma \ref{lem:product} we have that $F$ lies in $A(v)$ if and only if the residue class of $F$ modulo $\Phi$ is in $C$ and the residue class of $F$ modulo $(z-1)^2$ in $\Q$. The latter condition is equivalent to $F'(1)=0$, which implies the claim.
\end{proof}

\begin{example}
 In the case $v=(1,2,-3)$, by Example \ref{ex:12} the smallest $b$ such that $F_{a,b}$ can possibly be divisible by the polynomial $Q$ from Theorem \ref{thm:main} for some $a\in\N$ is $b=3\cdot 7\cdot 8=168$.
\end{example}

\section{Conclusion and Outlook}

As indicated already in the remarks after Question 20 in \cite{MR893156}, a similar strategy applies also to the cases $A=B=\SU(n)$. The only difference is that we now have to consider exponent vectors $(v_1,\dots,v_n) \in \Z^n$ satisfying $v_1+\cdots+ v_n=0$ and $\gcd(v_1,\dots,v_n)=1$ and Schur polynomials in $n$ variables. Putting it more precisely, we obtain the following question:
 
\begin{quest}
Given $v = (v_1,\dots,v_n)$ and $w = (w_1,\dots,w_n)$ be integer vectors satisfying $v_1+\cdots+v_n=w_1+\cdots+w_n=0$ and $\gcd(v_1,\dots, v_n) = \gcd(w_1,\dots,w_n) = 1$, can we find Schur positive symmetric polynomials in $n$ variables $P$ and $Q$ such that:
\begin{enumerate}
    \item $P_v(z) = Q_w(z)$,
    \item $P(\xi,\dots,\xi) \neq P(1,\dots,1)$ and $Q(\xi,\dots,\xi)\neq Q(1,\dots,1),$
    for $\xi^n=1,\xi \neq 1$.
\end{enumerate}
\end{quest}

Theorem~\ref{thm:main} describes the Laurent polynomials which are a $\Q$-linear combination of monomials in $(e_1)_v(z),(e_2)_v(z),(e_3)_v(z)$ where $e_k$ is the elementary symmetric polynomial in three variables of degree $k$. For our purposes it would however be much more desirable to have an understanding of which Laurent polynomials are a \emph{positive} $\Q$-linear combination of monomials in $(e_1)_v(z),(e_2)_v(z),(e_3)_v(z)$. Indeed, since elementary symmetric polynomials are Schur positive and since the product of Schur positive polynomials is again Schur positive, a positive $\Z$-linear combination of monomials in $e_1,e_2,e_3$ is always Schur positive. We tried to get our hands on this by suitable variants of P\'olya's Theorem \cite[Theorem~5.5.1]{MarshSOS} but we have not been successful.

\end{document}